\documentclass[12pt]{amsart}
\usepackage{a4wide}
\usepackage{url}
\usepackage{hyperref}
\usepackage{amsrefs}
\usepackage{amsthm}
\usepackage{float}
\usepackage{amsmath, amssymb, mathtools}
\usepackage{graphicx}
\usepackage{comment}
\usepackage{xcolor}

\newtheorem{theorem}{Theorem}[section]
\newtheorem{lemma}[theorem]{Lemma}

\newtheorem*{theorem*}{Theorem}

\theoremstyle{definition}

\theoremstyle{remark}
\newtheorem{remark}[theorem]{Remark}

\numberwithin{equation}{section}

\begin{document}

\title[Duality of ZMC surfaces in the Lorentzian Heisenberg group]
 {Duality of zero mean curvature surfaces in the Lorentzian Heisenberg group}

\author[S.R.R. Mohanty]{Sai Rasmi Ranjan Mohanty}
\address{Department of Mathematics, Shiv Nadar University, Dadri 201314, Uttar Pradesh, India}
\email{sairasmiranjan1996@gmail.com, sm743@snu.edu.in}

\author[P. Vasu]{Priyank Vasu}
\address{Department of Mathematics, Indian Institute of Technology Patna,
Patna-801106, Bihar, India}
\email{priyank\_2121ma16@iitp.ac.in}

\subjclass{Primary 53C42, 53A10; Secondary 53C50.}

\keywords{Heisenberg groups, maximal surfaces, timelike minimal surfaces, zero mean curvature surfaces, surface duality}



\begin{abstract}

We study a transformation surface associated with a zero mean curvature surface in the three-dimensional Heisenberg group with respect to two left-invariant semi-Riemannian metrics. We investigate the duality and prove that the transformation surface also has zero mean curvature. Furthermore, we derive the Sym formula for the dual surface in both metric cases.

\end{abstract}

\maketitle
\section{Introduction}

In classical surface theory in Euclidean space, certain minimal surfaces admit a
nontrivial partner surface, called a \emph{dual surface}, which is conformally related to the original surface and has parallel tangent planes.
A fundamental example is provided by the catenoid and the helicoid, which arise as a dual pair within a one-parameter family of associated surfaces.
Motivated by this, Christoffel in \cite{Christoffel1867} classified surfaces admitting a dual surface; these are precisely the \emph{isothermic surfaces}, and the corresponding dual is known as the Christoffel dual. The dual surface is unique up to homothety; furthermore, applying the duality transformation twice recovers the original surface. Dual surfaces in Euclidean space were studied by Kamberov et al.~\cite{KamberovPeditPinkall1998}; explicit representations of Christoffel duals in Euclidean and Minkowski spaces were obtained by Fujimori et al.~\cite{FujimoriHertrichJerominKokubuUmeharaYamada2018}, and a dual surface theory for minimal surfaces in the 3-dimensional Heisenberg group with left-invariant Riemannian metric was given by Kobayashi in \cite{Kobayashi2025Duality}.

In this paper, we investigate an analogue of dual surfaces in the three-dimensional
Heisenberg group $\mathrm{Nil}^3$ which is identified with $\mathbb{R}^3(x_1,x_2,x_3)$ equipped
with the group multiplication
\[
(x_1,x_2,x_3)\cdot(\tilde x_1,\tilde x_2,\tilde x_3)
= \bigl(x_1+\tilde x_1,\; x_2+\tilde x_2,\;
x_3+\tilde x_3+\tfrac{1}{2}(x_1\tilde x_2-\tilde x_1 x_2)\bigr),
\]
and endowed with either of the two left-invariant semi-Riemannian metrics
\[
ds_{\pm}^2 = \pm dx_1^2 + dx_2^2 \mp
\bigl(dx_3 + \tfrac12(x_2\,dx_1 - x_1\,dx_2)\bigr)^2.
\]
We denote $\operatorname{Nil}_1^3 := \bigl(\operatorname{Nil}^3, ds_+^2\bigr),
\ \operatorname{Nil}_t^3 := \bigl(\operatorname{Nil}^3, ds_-^2\bigr),$
and refer to them as the \emph{Lorentzian Heisenberg groups}. We introduce the notion of a dual surface associated with a zero mean
curvature surface in both $\operatorname{Nil}_1^3$ and $\operatorname{Nil}_t^3$, and show that the dual surface also has zero mean curvature and applying the dual transformation twice gives the original surface.
The main idea of this article is motivated by the work of Kobayashi~\cite{Kobayashi2025Duality}.

A surface immersed in a Lorentzian manifold is called \emph{spacelike} if the induced metric is Riemannian.
Spacelike zero mean curvature surfaces, also known as maximal surfaces, in
$\operatorname{Nil}_1^3$ have been studied extensively; see
\cites{Brander&Kobayashi2025,CintraMercuriOnnis2017,Lee2011}.
Analogously, a surface immersed in a Lorentzian manifold is called \emph{timelike} if the induced metric is Lorentzian.
Timelike zero mean curvature surfaces, or timelike minimal surfaces, in
$\operatorname{Nil}_t^3$ have been investigated in
\cites{Kiyohara&Kobayashi_timelike2022,CintraMercuriOnnis2017,Kiyohara2025}.


This paper is divided into two parts. First,
in Section~\ref{sec:Dual_maximal_surfaces}, we study maximal surfaces in
$\operatorname{Nil}_1^3$.
Subsection~\ref{subsec:Preliminaries_maximal_surfaces} recalls the spinor representation and the normal Gauss map of a maximal surface, and in Subsection~\ref{subsec: Dual maximal surface}, we define the dual generating spinors for a given maximal surface and
establish the maximality and duality of the associated surface
(Theorem~\ref{maximal main theorem}).
In Subsection~\ref{subsec: Sym formula for dual maximal surface}, we derive the
Sym formula for the dual maximal surface via the logarithmic derivative of the extended frame of the given maximal surface with respect to the spectral parameter. In addition, we show that the generalized maximal surface and its dual preserve both the singularity set and the nature of the singularities.
Analogously, in the second part, in Section~\ref{sec:Dual_timelike_minimal_surfaces},
we study timelike minimal surfaces in $\operatorname{Nil}_t^3$. In Subsection~\ref{subsec:Preliminaries_timelike_minimal_surfaces}, we begin by reviewing the
para-complex structure, the Gauss map, and the spinor representation of a timelike
minimal surface in $\operatorname{Nil}_t^3$.
In Subsection~\ref{subsec: Dual timelike minimal surface}, we define the dual generating
spinors and establish the main result of this section
(Theorem~\ref{thm 1 for timelike}).
Finally, in Subsection~\ref{subsec: The Sym formula for dual timelike minimal surface}, we derive the Sym formula for dual timelike minimal surfaces and conclude with
some examples.

\section{Dual maximal surfaces in \texorpdfstring{$\operatorname{Nil}_1^3$}{Nil_1^3} and the Sym formula}
 \label{sec:Dual_maximal_surfaces}

\subsection{Preliminaries for maximal surfaces}\label{subsec:Preliminaries_maximal_surfaces}
In this subsection, we recall the necessary background on maximal surfaces in the Lorentzian Heisenberg group $\operatorname{Nil}_1^3$ from \cite{Brander&Kobayashi2025}. In particular, we review the generating spinors associated with maximal surfaces and describe their normal Gauss maps. 

Let $f \colon M \longrightarrow \mathrm{Nil}_1^3$ be a conformal spacelike immersion of a Riemann surface $M$, equipped with a local conformal parameter $z = x + iy$ in a simply connected domain $\mathbb{D}$ in $M$, and an orthonormal basis of Lie algebra $\mathfrak{nil}_1^3$ of $\mathrm{Nil}_1^3$ is $\{e_1,e_2,e_3\}$. The left-invariant Maurer-Cartan form of $f$ is given by
\[
f^{-1}df = (f^{-1}f_z)\,dz + (f^{-1}f_{\bar z})\,d\bar z,
\]
where
\[
f^{-1}f_z = \sum_{k=1}^3 \varphi_k e_k,
\qquad
f^{-1}f_{\bar z} = \sum_{k=1}^3 \overline{\varphi_k}\, e_k.
\]
The conformality condition on $f$ yields the relations
\begin{equation}\label{eq:conformal_rel}
\varphi_1^2 + \varphi_2^2 - \varphi_3^2 = 0,
\qquad
|\varphi_1|^2 + |\varphi_2|^2 - |\varphi_3|^2 = \tfrac12 e^u \neq 0,
\end{equation}
and hence the induced metric is given by $I = e^u |dz|^2$.

The system \eqref{eq:conformal_rel} can be parametrized using a pair of \emph{generating spinors} $\psi_1$ and $\psi_2$, through the expressions
\begin{equation}\label{eq:phi_from_spinors}
\varphi_1 = (\overline{\psi_2})^2 - \psi_1^2, \qquad
\varphi_2 = i\big( (\overline{\psi_2})^2 + \psi_1^2 \big), \qquad
\varphi_3 = 2i \psi_1 \overline{\psi_2}.
\end{equation}

Let $N$ denote the positively oriented unit normal vector field along the immersion $f$ and define the rescaled normal vector $L:= e^{u/2} N$. The \emph{support function} $h$ is then defined by $
h = \langle f^{-1} L, e_3 \rangle.$
In terms of the generating spinors $\psi_1$ and $\psi_2$, we obtain the identities
\begin{equation}\label{eq:u_h_spinors}
e^u = 4\big(|\psi_1|^2 - |\psi_2|^2\big)^2,
\qquad
h = 2\big(|\psi_1|^2 + |\psi_2|^2\big).
\end{equation}
The immersion $f$ is regular if and only if $|\psi_1| \neq |\psi_2|$. Furthermore, we set the Dirac potential by setting $e^{w/2} = U = V$, and define the Abresch-Rosenberg quadratic differential by $Q\,dz^2 = 4B\,dz^2$. Following \cite{Kiyohara2025}, these can be expressed as
\begin{equation}\label{eq:UV_B}
e^{w/2} = U = V = \frac{H}{2} e^{u/2} i - \frac{1}{4} h,
\qquad
B = \frac{2iH + 1}{2}
\bigl( \psi_1 \overline{\psi_{2}}_z - \overline{\psi_2} \, {\psi_1}_z \bigr)
+ 2iH (\psi_1 \overline{\psi_2})^2,
\end{equation}
where $H$ is the mean curvature of the spacelike surface $f$.

The generating spinors $\tilde{\psi} = (\psi_1, \psi_2)$ satisfy a \emph{linear spinor system} of the form
\begin{equation}\label{eq:linear_spinor_system}
\tilde{\psi}_z = \tilde{\psi} \, \widetilde{U},
\qquad
\tilde{\psi}_{\bar z} = \tilde{\psi} \, \widetilde{V},
\end{equation}
where the coefficient matrices are given by
\[
\widetilde{U} =
\begin{pmatrix}
\frac12 w_z - \frac{i}{2} H_z e^{-w/2 + u/2} & - e^{w/2} \\[4pt]
B e^{-w/2} & 0
\end{pmatrix},
\qquad
\widetilde{V} =
\begin{pmatrix}
0 & - \overline{B} e^{-w/2} \\[4pt]
e^{w/2} & \frac12 w_{\bar z} - \frac{i}{2} H_{\bar z} e^{-w/2 + u/2}
\end{pmatrix}.
\]

The second column of the first equation, together with the first column of the second equation, yields the nonlinear Dirac equations. In particular, letting $\mathfrak{D}$ denote the Dirac operator for the conformal immersion $f$ in $\operatorname{Nil_1^3}$,
\begin{equation}\label{eq:dirac}
\mathfrak{D}\begin{pmatrix}\psi_1\\\psi_2\end{pmatrix}
:=\begin{pmatrix}
\partial_z\psi_2 + U\psi_1\\[4pt]
-\partial_{\bar z}\psi_1 + V\psi_2
\end{pmatrix}
=
\begin{pmatrix}0\\0\end{pmatrix},
\qquad U=V=e^{w/2}.
\end{equation}

The left-translated unit normal vector $f^{-1}N$ then takes values in the union of the two hyperbolic planes
$\mathbb{H}_+^2 \cup \mathbb{H}_-^2 \subset \mathfrak{nil}_1^3$.
More precisely, $f^{-1}N$ lies in $\mathbb{H}_+^2$ when $|\psi_1| < |\psi_2|$, and in $\mathbb{H}_-^2$ when $|\psi_1| > |\psi_2|$.
By composing $f^{-1}N$ with the stereographic projection $\pi$ from the south pole, we obtain the \emph{normal Gauss map} $
g = \pi \circ (f^{-1}N) \colon M \longrightarrow \mathbb{C} \cup \{\infty\} \setminus \mathbb{S}^1.$ In terms of the generating spinors, the Gauss map $g$ and the normal vector $f^{-1}N$ can be written as
\begin{equation}\label{eq:gauss_spinor}
g = \frac{\psi_1}{\overline{\psi_2}},
\qquad
f^{-1}N =
\frac{1}{|g|^2 - 1}
\bigl(
2\,\Im(g)\,e_1 - 2\,\Re(g)\,e_2 - (|g|^2 + 1)\,e_3
\bigr).
\end{equation}

We next recall the associated family of Maurer-Cartan forms $\alpha^\lambda$, defined for $\lambda \in \mathbb{S}^1$ by
\begin{equation}\label{eq:maure_Cartan_H_nonzero}
\alpha^\lambda = \widetilde{U}^\lambda\,dz + \widetilde{V}^\lambda\,d\bar z.
\end{equation}
Here the matrices $\widetilde{U}_\lambda$ and $\widetilde{V}_\lambda$ are given by
\begin{equation}\label{eq:UV_tilde_H_nonzero}
\widetilde{U}^\lambda =
\begin{pmatrix}
\frac14 w_z - \frac{i}{2} H_z e^{-w/2+u/2} & -\lambda^{-1} e^{w/2} \\[6pt]
\lambda^{-1} B e^{-w/2} & -\frac14 w_z
\end{pmatrix},
\quad
\widetilde{V}^\lambda =
\begin{pmatrix}
-\frac14 w_{\bar z} & -\lambda \overline{B} e^{-w/2} \\[6pt]
\lambda e^{w/2} & \frac14 w_{\bar z} - \frac{i}{2} H_{\bar z} e^{-w/2+u/2}
\end{pmatrix}.
\end{equation}
Note that $\widetilde U^\lambda|_{\lambda=1}=\widetilde U$ and
$\widetilde V^\lambda|_{\lambda=1}=\widetilde V$ in \eqref{eq:linear_spinor_system} after suitable gauge transformation applied to $\widetilde{U}$ and $\widetilde{V}$.

\begin{theorem} \cite[Theorem 3.2]{Brander&Kobayashi2025}\label{thm:3.2_brander_maximal}
Let $f:\mathbb{D}\longrightarrow\mathrm{Nil_1^3}$ be a conformal spacelike immersion and $\alpha_\lambda$
the $1$-form defined in \eqref{eq:maure_Cartan_H_nonzero} and $g$ is the normal Gauss map. Then the following statements are equivalent:
\begin{enumerate}
\item $f$ is a maximal surface.
\item $d + \alpha^\lambda$ is a family of flat connections on $\mathbb D\times\mathrm{SU}_{2}$.
\item The normal Gauss map $g$ for $f$ is a nowhere holomorphic harmonic map into the
2-sphere.
\end{enumerate}
\end{theorem}
For more details on the maximal surface in $\mathrm{Nil_1^3}$, see \cite{Brander&Kobayashi2025,Lee2011}.

\medskip
Next, starting from a given pair of generating spinors associated with a maximal surface in $\mathrm{Nil}_1^3$, we introduce a new pair of \emph{dual generating spinors}.
We show that this dual pair satisfies the corresponding spinor system and therefore
determines a new maximal surface in $\mathrm{Nil}_1^3$.
Moreover, applying the same dual construction to the new surface recovers the original
surface, up to rigid motions.

\subsection{Dual maximal surface}\label{subsec: Dual maximal surface}
Let $\psi_1$ and $\psi_2$ denote the generating spinors of a conformal maximal surface
$f$ in $\mathrm{Nil}_1^3$. We define the corresponding \emph{dual generating spinors}
$\psi_1^{\ast}$ and $\psi_2^{\ast}$ by
\begin{equation}\label{dual_spinors}
\psi_1^{\ast}
= \frac{4\sqrt{-B}}{h}\,\psi_2,
\qquad
\psi_2^{\ast}
= \frac{4\sqrt{-\overline{B}}}{h}\,\psi_1,
\end{equation}
where $B$ is the Abresch-Rosenberg differential and $h$ is the support function of $f$ as
introduced in Subsection~\ref{subsec:Preliminaries_maximal_surfaces}.
In the case $B \equiv 0$, we have $\psi_1^{\ast} \equiv \psi_2^{\ast} \equiv 0$; therefore, we exclude such maximal surfaces. If $B \not\equiv 0$, then since $B$ is holomorphic its zeros are isolated and correspond to
branch points of the dual surface, which will be discussed in
Theorem~\ref{maximal main theorem}.

The following lemma shows that the dual generating spinors satisfy a linear spinor system.

\begin{lemma}\label{lemma:maximal}
 Let $f \colon \mathbb{D} \longrightarrow \mathrm{Nil}_1^3$ be a conformal maximal surface with generating
spinors $(\psi_1, \psi_2)$, and suppose that $B\not\equiv0$.
Let $(\psi_1^{\ast}, \psi_2^{\ast})$ denote the associated dual generating spinors
defined in \eqref{dual_spinors}.
Then the vector of generating spinors $\tilde{\psi}^{\ast} = (\psi_1^{\ast}, \psi_2^{\ast})$ satisfies the following linear spinor system:
\begin{equation}\label{eq:spinor_system_maximal}
  \tilde\psi^{\ast}_z = \tilde\psi^{\ast}\,\widetilde U^{\ast},
  \qquad
  \tilde\psi^{\ast}_{\bar z} = \tilde\psi^{\ast}\,\widetilde V^{\ast},
\end{equation}
where
\[
\widetilde U^{\ast} =
\begin{pmatrix}
\dfrac12 w^{\ast}_z & -e^{w^{\ast}/2}\\[3pt]
B^{\ast}e^{-w^{\ast}/2} & 0
\end{pmatrix},
\qquad
\widetilde V^{\ast} =
\begin{pmatrix}
0 & -\overline{B^{\ast}}e^{-w^{\ast}/2}\\[3pt]
e^{w^{\ast}/2} & \dfrac12 w^{\ast}_{\bar z}
\end{pmatrix},
\]
with
\begin{equation}\label{eq:B&Dicar_potential_maximal}
e^{w^{\ast}/2} = \frac{4|B|}{h},\qquad B^{\ast}=B.
\end{equation}
\end{lemma}
\begin{proof}
   Using \eqref{eq:B&Dicar_potential_maximal},  and the relations
\[
\frac12 w^{\ast}_z
   = \Bigl(\log{\frac{4\sqrt{-B}}{h}}\Bigr)_z,
   \qquad
   \frac{ h B}{4|B|}
   = B^{\ast} e^{-w^{\ast}/2}, \qquad e^{w/2}=- h/4,
\]
where $B(z)\neq 0$, we compute
\begin{align*}
\psi^{\ast}_{1z}
 &=\Bigl({\frac{4\sqrt{-B}}{h}}\Bigr)_z\psi_2
   +{\frac{4\sqrt{-B}}{h}}\;\psi_{2z}\\
 &=\Bigl(\log{\frac{4\sqrt{-B}}{h}}\Bigr)_z\psi_1^{\ast}
  +\frac{ h B}{4|B|}\psi_2^{\ast}\\
  &=\frac 12 w_z^\ast\psi_1^\ast+B^\ast e^{-w^\ast/2}\psi_2^\ast.
\end{align*}
  Similarly, 
\begin{align*}
\psi^{\ast}_{2z}
 &=\Bigl(\frac{4\sqrt{-\overline{B}}}{h}\Bigr)_z\psi_1
   +\frac{4\sqrt{-\overline{B}}}{h}\;\psi_{1z}\\
 &=-\!(\log h)_z\frac{4\sqrt{-\overline{B}}}{h}\psi_1
   +\frac{4\sqrt{-\overline{B}}}{h}\Bigl(\tfrac12 w_z\psi_1+Be^{-w/2}\psi_2\Bigr)\\
   &=-e^{w^{\ast}/2}\psi^{\ast}_1.
\end{align*}
Furthermore, a similar computation for $\tilde\psi^{\ast}_{\bar z}$ gives the second equation
in~\eqref{eq:spinor_system_maximal}.  
\end{proof}

Using the above lemma, we now state the main theorem for this section.

\begin{theorem}\label{maximal main theorem}
Retain all the notation and the assumptions of Lemma~\ref{lemma:maximal} and the relation in \eqref{eq:B&Dicar_potential_maximal}, there exists a maximal immersion $f^{\ast}$ in $\mathrm{Nil_1^3}$ whose generating
spinors are $\psi_1^{\ast}$ and $\psi_2^{\ast}$, with Dirac potential
$U^{\ast}=V^{\ast}=e^{w^{\ast}/2}$, and whose left-translated unit normal equals
$f^{-1}N$ up to sign.  
Moreover, the following hold:
\begin{enumerate}
\item By choosing the left-translated unit normal $N^{\ast}$ so that $(f^{\ast})^{-1}N^{\ast}=f^{-1}N\subset\mathfrak{nil_1^3}$, the metric $e^{u^{\ast}}|dz|^2$, support function $h^{\ast}$, Abresch-Rosenberg differential $B^{\ast}dz^2$, and normal Gauss map $g^{\ast}$ of $f^{\ast}$ are respectively given by
      \begin{equation}\label{eq:ehBg_dual_maximal}
      e^{u^{\ast}}
        = \frac{16^2|B|^2}{h^4}\,e^u,
      \qquad
      h^{\ast} = \frac{16|B|}{h},
      \qquad
      B^{\ast}=B,
      \qquad\text{and }\quad
      g^{\ast}={g}.
      \end{equation}
\item It satisfies the duality $f^{\ast\ast}=f$ up to a rigid motion, and the zeros of $B$ are branch points of the dual $f^\ast$.

\end{enumerate}
The surface $f^{\ast}$ will be called the \emph{dual maximal surface} to $f$.
\end{theorem}

\begin{proof}
  
By using Lemma~\ref{lemma:maximal}, the vector of generating spinors $\tilde\psi^{\ast}=(\psi_1^{\ast},\psi_2^{\ast})$ defines a conformal immersion
$f^{\ast}$ in $\mathrm{Nil_1^3}$.  The Dirac potential
\[
U^{\ast}=V^{\ast}=e^{w^{\ast}/2}
  = \frac{H^{\ast}}{2}ie^{u^{\ast}/2}
    -\frac{1}{4}h^{\ast}
  = \frac{4|B|}{h}\in \mathbb{R}.
\]
Thus $f^{\ast}$ is maximal surface in $\operatorname{Nil_1^3}$.
Moreover, the tangent planes of $f$ and $f^{\ast}$ coincide, so their left-translated unit normals $f^{-1}N$ and $(f^{\ast})^{-1}N^{\ast}$ agree
up to sign.

(1) The metric follows from the definition of
$\psi_1^{\ast}$ and $\psi_2^{\ast}$, and $h^\ast$ follows from the definition of support function of $f^\ast$. Since we choose $(f^{\ast})^{-1}N^{\ast}=f^{-1}N$, therefore $g^{\ast}=g$. 

(2) The double-dual spinors are
$$
\psi^{\ast\ast}_1
 = {\frac{4\sqrt{-B^{\ast}}}{h^{\ast}}}\,\psi^{\ast}_2=\psi_1,
\qquad
\psi^{\ast\ast}_2
 = {\frac{4\sqrt{-B^{\ast}}}{h^{\ast}}}\,\psi^{\ast}_1=\psi_2.
$$
Hence $(\psi^{\ast\ast}_1,\psi^{\ast\ast}_2)$ defines the original surface $f$ up to rigid
motion, and since $e^{u^{\ast}}
        = \frac{16^2|B|^2}{h^4}\,e^u$, it is clear that the zeros of $B$ are branch points of $f^\ast$.
\end{proof}


In the following subsection, we present the Sym formula for dual maximal surfaces in
$\mathrm{Nil}_1^3$. Our approach relies on the loop group method developed for the study
of maximal surfaces in $\mathrm{Nil}_1^3$. (see \cite{Brander&Kobayashi2025})

\subsection{Sym formula for dual maximal surface} \label{subsec: Sym formula for dual maximal surface}
We begin by identifying the Lie algebra $\mathfrak{nil}_1^3$ with $\mathfrak{su}_2$ (Lie algebra of $\operatorname{SU}_2$) as
real vector spaces. In $\mathfrak{su}_2$, we fix the basis
\begin{equation}\label{eq:basis_su_2}
E_1=\frac12\begin{pmatrix}0&-i\\-i&0\end{pmatrix},\qquad
E_2=\frac12\begin{pmatrix}0&1\\-1&0\end{pmatrix},\qquad
E_3=\frac12\begin{pmatrix}i&0\\0&-i\end{pmatrix},
\end{equation}
which forms an orthogonal basis of $\mathfrak{su}_2$.
Then a linear isomorphism $\Xi \colon \mathfrak{su}_2 \longrightarrow \mathfrak{nil}_1^3$ is given by $x_1 E_1 + x_2 E_2 + x_3 E_3 \longmapsto x_1 e_1 + x_2 e_2 + x_3 e_3.$ Define a smooth bijection $\Xi_{\mathrm{nil}}:\mathfrak{su}_{2}\longrightarrow\mathrm{Nil_1^3}$ by
$\Xi_{\mathrm{nil}}:=\exp\circ\Xi$;
$
\Xi_{\mathrm{nil}}(x_1E_1 + x_2E_2 + x_3E_3)=(x_1, x_2, x_3).
$

Now by Theorem~\ref{thm:3.2_brander_maximal}, the associated family of
Maurer-Cartan forms $\alpha^\lambda$ in \eqref{eq:maure_Cartan_H_nonzero} for maximal surfaces,
can be simplified as follows:
\begin{equation}\label{eq:maure_Cartan_maximal}
\alpha^\lambda = U^\lambda\,dz + V^\lambda\,d\bar z,
\end{equation}
with
\begin{equation}\label{eq:UV_Maurer_Cartan_maximal}
U^\lambda =
\begin{pmatrix}
\frac14 w_z & -\lambda^{-1} e^{w/2}\\[6pt]
\lambda^{-1} B e^{-w/2} & -\frac14 w_z
\end{pmatrix},
\qquad
V^\lambda =
\begin{pmatrix}
-\frac14 w_{\bar z} & -\lambda\overline{B} e^{-w/2}\\[6pt]
\lambda e^{w/2} & \frac14 w_{\bar z}
\end{pmatrix}.
\end{equation}
Let $F$ be the
corresponding $\mathrm{SU}_{2}$ valued solution of the equation $F^{-1}dF=\alpha^\lambda$
$(\lambda\in \mathbb S^1)$, where $\alpha^\lambda$ is defined by \eqref{eq:maure_Cartan_maximal}. We call $F$
the \emph{extended frame of the maximal surface} $f$. In particular, $F$ can be represented as:
\begin{equation}\label{eq:extendedframe_maximal}
F(\lambda)=\frac{1}{\sqrt{| \psi_1(\lambda)|^2 + | \psi_2(\lambda)|^2}}
\begin{pmatrix}
\psi_1(\lambda) & \psi_2(\lambda)\\[6pt]
-\overline{\psi_2(\lambda)} & \overline{\psi_1(\lambda)}
\end{pmatrix},
\end{equation}
where $\psi_1(\lambda)$ and $\psi_2(\lambda)$ denote families of functions such that
$\psi_1(\lambda)|_{\lambda=1}=\psi_1$ and $\psi_2(\lambda)|_{\lambda=1}=\psi_2$ are the generating spinors
of the maximal surface $f$.

We now prove the Sym formula.

\begin{theorem}\label{thm:Sym_formula_dualmaximal}
Let $F$ be the extended frame for some maximal surface $f$ on $\mathbb D$ with
$B\not\equiv0$ in $\mathrm{Nil_1^3}$, and define $m_\pm$ and $N_m$ respectively by
$$
m_\pm = \pm\,i\,\lambda(\partial_\lambda F)F^{-1}- N_m,\qquad
N_m=\tfrac{i}{2}\operatorname{Ad}(F)\sigma_3,
$$
where $\sigma_3=\begin{pmatrix}
1 & 0\\
0 & -1
\end{pmatrix}$. Moreover, define maps $f^\lambda_\pm:\mathbb D\longrightarrow\mathrm{Nil_1^3}$ by $f^\lambda_\pm:=\Xi_{\mathrm{nil}}\circ\widehat f^\lambda_\pm$
with
\begin{equation}\label{eq:sym_formula_maximal}
\widehat f^\lambda_\pm=
\Bigl(m_\pm^{\mathrm{o}} - \tfrac{i}{2}\lambda(\partial_\lambda m_\pm)^{\mathrm{d}}\Bigr)\Big|_{\lambda\in \mathbb S^1},
\end{equation}
where superscripts ``o'' and ``d'' denote off-diagonal and diagonal parts, respectively.
Then, for each $\lambda\in \mathbb S^1$, the maps $f^\lambda_\pm$ are maximal surfaces in $\mathrm{Nil_1^3}$
and $N_m$ is the normal Gauss map of $f^\lambda_\pm$. In particular, $f^\lambda_-|_{\lambda=1}$
is the original maximal surface $f$ up to a rigid motion and $f^\lambda_+|_{\lambda=1}$ is the
dual maximal surface $f^\ast$ of $f$. Moreover, for each $\lambda\in \mathbb S^1$, $f^\lambda_-$ and $f^\lambda_+$ are dual to each other.
\end{theorem}

\begin{proof}
It is proven that $f^\lambda_-|_{\lambda=1}$ is a maximal surface in $\operatorname{Nil_1^3}$ in \cite[Theorem 3.4]{Brander&Kobayashi2025}.
We only show for $f^\lambda_+|_{\lambda=1}$, and both are dual to each other for each $\lambda\in \mathbb S^1$. From a direct calculation, we have
$$
\partial_z m_+ = \operatorname{Ad}(F)\Bigl(i\lambda\partial_\lambda U_\lambda - \frac{i}{2}[U^\lambda,\sigma_3]\Bigr)
= -2i\lambda^{-1} B e^{-w/2}\operatorname{Ad}(F)\begin{pmatrix}0&0\\1&0\end{pmatrix}.
$$
Using the generating spinors and the support function $h$, this becomes
\[
\partial_z m_+ = \frac{16i\lambda^{-1}B}{h^2}
\begin{pmatrix}
\psi_2(\lambda)\overline{\psi_1(\lambda)} & -\psi_2(\lambda)^2\\[4pt]
\overline{\psi_1(\lambda)^2} & -\psi_2(\lambda)\overline{\psi_1(\lambda)}
\end{pmatrix}.
\]
From the dual
generating spinors as in \eqref{dual_spinors},
we have
\begin{equation}\label{eq:partial_m+_maximal}
\partial_z m_+ = \varphi_1^+(\lambda)E_1 + \varphi_2^+(\lambda)E_2 + i\,\varphi_3^+(\lambda)E_3,
\end{equation}
where
\begin{align*}
\varphi_1^+(\lambda)&=\lambda^{-1}\bigl((\overline{\psi_2}^\ast(\lambda))^2 - (\psi_1^\ast(\lambda))^2\bigr),\\
\varphi_2^+(\lambda)&=i\lambda^{-1}\bigl((\overline{\psi_2}^\ast(\lambda))^2 + (\psi_1^\ast(\lambda))^2\bigr),\\
\varphi_3^+(\lambda)&=2i\lambda^{-1}\psi_1^\ast(\lambda)\overline{\psi_2}^\ast(\lambda).
\end{align*}
A computation of $\partial_z(i\lambda\partial_\lambda m_+)$ combined with the diagonal part
analysis yields
\begin{equation}\label{eq:2ndpart_symformula_maximal}
\begin{aligned}
\partial_z\Bigl(-\tfrac{i}{2}\lambda(\partial_\lambda m_+)\Bigr)^d
&=\frac{i}{2}(\partial_z m_+)^d-\frac{1}{2}[m_++N_m,\,\partial_z m_+]^d\\
&=i(\partial_z m_+)^d-\frac{1}{2}[m_+,\,\partial_z m_+]^d\\
&=-\Bigl(\varphi_3^+(\lambda) - \frac12\varphi_1^+(\lambda)\int\varphi_2^+(\lambda)\,dz
 + \frac12\varphi_2^+(\lambda)\int\varphi_1^+(\lambda)\,dz\Bigr)E_3.
 \end{aligned}
\end{equation}
Combining \eqref{eq:partial_m+_maximal} and \eqref{eq:2ndpart_symformula_maximal} we obtain
\[
\partial_z\widehat f^\lambda_+ = \varphi_1^+(\lambda)E_1 + \varphi_2^+(\lambda)E_2
- \Bigl(\varphi_3^+(\lambda) - \frac12\varphi_1^+(\lambda)\int\varphi_2^+(\lambda)\,dz
 + \frac12\varphi_2^+(\lambda)\int\varphi_1^+(\lambda)\,dz\Bigr)E_3.
\]
Using the identification \eqref{eq:basis_su_2} and the left translation by $(f^\lambda_+)^{-1}$, i.e.
$$
(f^\lambda_+)^{-1}\partial_z f^\lambda_+ = \varphi_1^+(\lambda)e_1 + \varphi_2^+(\lambda)e_2 + \varphi_3^+(\lambda)e_3,
$$
we see that $\psi_1^\ast(\lambda)$ and $\psi_2^\ast(\lambda)$ are spinors for $f^\lambda_+$ for
each $\lambda\in \mathbb S^1$. In particular, the function
\[
-\frac{1}{2}\bigl(|\psi_1^\ast(\lambda)|^2 + |\psi_2^\ast(\lambda)|^2\bigr) = e^{w^\ast/2} = \frac{4|B|}{h}
\]
is independent of $\lambda$, so the mean curvature $H=0$. Moreover, the conformal factor of the induced metric of $f^\lambda_+$ is given by
\[
e^{u^\ast}=4\bigl(|\psi_1^\ast(\lambda)|^2 - |\psi_2^\ast(\lambda)|^2\bigr)^2
= \frac{16^2|B|^2}{h^4}\,e^u.
\]
Thus, for each $\lambda \in \mathbb{S}^1$ with $B(z)\neq 0$,
the immersion $f^\lambda_+$ defines a maximal surface in $\mathrm{Nil}_1^3$.

\end{proof}

\begin{remark}
It can be observed that replacing $m_+$ by $-m_+$ in Theorem~\ref{thm:Sym_formula_dualmaximal} also gives a dual to the given maximal surface, and the two dual surfaces are the same up to rigidity.
\end{remark}

\textbf{Example:} Consider two hyperbolic paraboloids $x_3=-\tfrac{x_1x_2}{2}$ and $x_3=\tfrac{x_1x_2}{2}$  and their corresponding holomorphic potentials defined as
\[
\zeta_1 = \lambda^{-1}
\begin{pmatrix}
0 & -1 \\
1 & 0
\end{pmatrix} dz,\quad \text{ and }\quad\zeta_2 = \lambda^{-1}
\begin{pmatrix}
0 & 1 \\
1 & 0
\end{pmatrix} dz.
\]

Let $\Phi_1$ and $\Phi_2$ be the solutions of the ODEs $d\Phi_1=\Phi_1\zeta_1$ and $d\Phi_2=\Phi_2\zeta_2$ with the initial conditions $\Phi_1(z=0)=\mathrm{id}$ and $\Phi_2(z=0)=\mathrm{id}$, respectively.
Then,  the extended frames $F_1$ and $F_2$ from the Iwasawa decompositions $\Phi_1=F_1V_1$ and $\Phi_2=F_2V_2$, respectively, are given by
\begin{align*}
F_1 &=
\begin{pmatrix}
\cos (\lambda^{-1}z+\lambda\bar z) &
 -\sin (\lambda^{-1}z+\lambda\bar z) \\[6pt]
 \sin (\lambda^{-1}z+\lambda\bar z) &
\cos (\lambda^{-1}z+\lambda\bar z
\end{pmatrix},\\
F_2 &=
\begin{pmatrix}
\cosh (\lambda^{-1}z-\lambda\bar z) &
 \sinh (\lambda^{-1}z-\lambda\bar z) \\[6pt]
 \sinh (\lambda^{-1}z-\lambda\bar z) &
\cosh (\lambda^{-1}z-\lambda\bar z
\end{pmatrix}.
\end{align*}

By applying Sym formula as in \eqref{eq:sym_formula_maximal} to both extended frames $F_1$ and $F_2$, the respective family of maximal surfaces $f_{1-}^\mu$ and $f_{2-}^\mu$  in $\mathrm{Nil}_1^3$ are given by
\begin{align*}
 f_{1-}^\lambda & =
\bigg(
 \sin (2(\lambda^{-1}z+\lambda\bar z)),\,-2i(\lambda^{-1}z-\lambda\bar z),
\, i(\lambda^{-1}z-\lambda\bar z)
\sin (2i(\lambda^{-1}z+\lambda\bar z))
\bigg),\\
f_{2-}^\lambda & =
\bigg(
-2(\lambda^{-1}z+\lambda\bar z),\, \sinh (2i(\lambda^{-1}z-\lambda\bar z)),
\, -(\lambda^{-1}z+\lambda\bar z)
\sinh (2i(\lambda^{-1}z-\lambda\bar z))
\bigg).
\end{align*}
For each $\mu=e^{i't}\in\mathbb S_1^1,$ $f_{1-}^\mu$ and $f_{2-}^\mu$ are parts of the hyperbolic paraboloids $x_3=-\tfrac{x_1x_2}{2}$ and $x_3=\tfrac{x_1x_2}{2}$ respectively. From Sym formula as in \eqref{eq:sym_formula_maximal}, the corresponding dual surfaces are
\begin{align*}
    f_{1+}^\lambda & =
\bigg(
 \sin (2(\lambda^{-1}z+\lambda\bar z)),\,2i(\lambda^{-1}z-\lambda\bar z),
\, i(\lambda^{-1}z-\lambda\bar z)
\sin (2i(\lambda^{-1}z+\lambda\bar z))
\bigg),\\
f_{2+}^\lambda & =
\left(
2(\lambda^{-1}z+\lambda\bar z),\, \sinh (2i(\lambda^{-1}z-\lambda\bar z)),
\, -(\lambda^{-1}z+\lambda\bar z)
\sinh (2i(\lambda^{-1}z-\lambda\bar z))
\right).
\end{align*}
Then  $f_{1+}^\mu$ and $f_{2+}^\mu$ are the parts of the  hyperbolic paraboloids $x_3=\tfrac{x_1x_2}{2}$ and $x_3=-\tfrac{x_1x_2}{2}$, respectively.

\subsection{Singularities for maximal surfaces}\label{subsec:sing_maximal}
It is natural to consider maximal surfaces in $\operatorname{Nil}_1^3$ that admit
singularities.
As shown in \cite{Brander&Kobayashi2025}, the induced metric $I$ of a maximal immersion
$f$, expressed in terms of the normal Gauss map $g$, can be written as
\[
I = 4\bigl(|\psi_1|^2 - |\psi_2|^2\bigr)^2 |dz|^2
  = 4\,|\bar g_z|^2 \frac{(1 - |g|^2)^2}{(1 + |g|^2)^4}\,|dz|^2.
\]
It follows that singularities occur precisely at points where either $|g| = 1$ or
$g_{\bar z} = 0$.
Moreover, points at which the Gauss map $g$ is holomorphic are always degenerate
singularities.
Any maximal surface with singularities whose normal Gauss map
$g \colon M \to \mathbb{C} \cup \{\infty\}$ is nowhere holomorphic is referred to as a
\emph{generalized maximal surface}.
Also note that the generic singularities of generalized maximal surfaces are cuspidal edges,
swallowtails, and cuspidal cross caps; see \cite{Brander&Kobayashi2025} for details.

 As in the preceding subsection, a generalized maximal surface can be constructed from generating spinors. From Theorem~\ref{maximal main theorem}, the maximal surface $f$ and its dual $f^\ast$ share
the same normal Gauss map (upon choosing the same left-translated unit normal) and the same Abresch-Rosenberg differential.
Consequently, analogous properties hold for generalized maximal surfaces. The correspondence between the generic singularities of a generalized maximal surface and those of its dual follows from Lemma~4.3 and Theorem~4.4 in \cite{Brander&Kobayashi2025}.
Hence, we obtain the following result.


\begin{theorem}
     Let $f$ be a generalized maximal surface with nonvanishing Abresch-Rosenberg differential, and $f^\ast$ be its dual surface. Then both surfaces $f$ and $f^\ast$ have the same singularity set and the same generic singularities.
\end{theorem}

\section{Dual Timelike minimal surfaces in $\operatorname{Nil_3}$ and its Sym formula}\label{sec:Dual_timelike_minimal_surfaces}

In this section, we will discuss the dual surfaces of the timelike minimal surfaces in $\mathrm{Nil}_t^3$. For the sake of simplicity, we retain the same notations as in the previous section. 

\subsection{Preliminaries for timelike minimal surfaces}\label{subsec:Preliminaries_timelike_minimal_surfaces}
In this subsection, we recall the necessary background on timelike minimal surfaces in the
Lorentzian Heisenberg group $\operatorname{Nil}_t^3$ from \cite{Kiyohara&Kobayashi_timelike2022}.
We begin with a brief review of para-complex numbers and then discuss the generating
spinors associated with timelike minimal surfaces, together with their normal Gauss maps.

The ring of para-complex numbers denoted by $\mathbb{C'}$  is defined as,
$$\mathbb{C'}:=\left\{ z = x + i'y  \mid x,y \in \mathbb{R},\,(i')^{2} = 1\right\}.$$
We define the real part, imaginary part, and para-conjugation by $\Re z = x, \Im z = y,\text{ and } \bar z = x - i'y.$
Also note that $z\bar{z}=0$ does not imply that $z=0$.
A para-complex number $z = x + i'y$ has a square root in $\mathbb C'$ if and only if
$x + y \ge 0,\, x - y \ge 0.$
In particular, $i'$ has no para-complex square root.  
Similarly, $z = e^{w}$ for some $w \in \mathbb C'$ holds if and only if
$x + y > 0, \, x - y > 0.$

Let $M$ be an orientable, connected $2$-manifold, and let $G$ be a Lorentzian manifold.  
A smooth immersion $f: M \longrightarrow G$ is called \emph{timelike} if the induced metric is Lorentzian. Locally, $M$ admits para-complex coordinates $z = x + i' y$ such that $I = e^{u}\, dz\, d\bar z = e^{u}\big( (dx)^2 - (dy)^2 \big),$
where $z$ is a conformal coordinate and $e^{u}$ the conformal factor. Also, for a para-complex coordinate $z = x + yi'$, we define
\[
\partial_{z} := \frac12(\partial_{x} + i'\partial_{y}), \qquad
\partial_{\bar z} := \frac12(\partial_{x} - i'\partial_{y}).
\]

Let $f:M\longrightarrow\mathrm{Nil}_t^3$ be a conformal immersion from a Lorentz surface $M$ with a local
conformal coordinate $z=x+yi'$ in a simply connected domain $\mathbb D'\subset M$, and orthonormal basis of Lie algebra $\mathfrak{nil}_t^3$ of $\mathrm{Nil}_t^3$ is $\{\varepsilon_1,\varepsilon_2,\varepsilon_3\}$. The left-translated Maurer-Cartan form expands as
\[
f^{-1}df = (f^{-1}f_z)\,dz + (f^{-1}f_{\bar z})\,d\bar z,
\qquad f^{-1}f_z=\sum_{k=1}^3\varphi_k \varepsilon_k,\quad f^{-1}f_{\bar z}=\sum_{k=1}^3\overline{\varphi_k}\varepsilon_k.
\]
The conformality of $f$ gives the algebraic relations
\begin{equation}\label{eq:conformal_rel_timelike}
 -\varphi_1^2+\varphi_2^2+\varphi_3^2 = 0, \qquad  -\varphi_1\overline{\varphi_1}+\varphi_2\overline{\varphi_2}+\varphi_3\overline{\varphi_3} = \tfrac12 e^u \neq 0,
\end{equation}
so the induced metric is $I=e^udzd\overline{z}$.

From \cite[Lemma 2.1]{Kiyohara&Kobayashi_timelike2022}, the system \eqref{eq:conformal_rel_timelike} can be parametrized using a pair of \emph{generating spinors} $\psi_1$ and $\psi_2$, through the expressions
\begin{equation}\label{eq:phi_from_spinors_timelike}
\varphi_1=\epsilon\bigg((\overline{\psi_2})^2+\psi_1^2\bigg),\qquad
\varphi_2=\epsilon i'\bigg((\overline{\psi_2})^2-\psi_1^2\bigg),\qquad
\varphi_3=2\psi_1\overline{\psi_2},
\end{equation}
where $\epsilon=i'$ or $-i'$.
Let $N$ denote the positively oriented unit normal vector field along the immersion $f$ and define the rescaled normal vector $L:= e^{u/2} N$. The \emph{support function} $h$ is then defined by $
h = \langle f^{-1} L, \varepsilon_3 \rangle.$
In terms of the generating spinors $\psi_1$ and $\psi_2$, we obtain the identities
\begin{equation}\label{eq:u_h_spinors for timelike}
e^u = 4\big(\psi_1\overline{\psi_1}+\psi_2\overline{\psi_2}\big)^2,\qquad
h = 2\big(\psi_2\overline{\psi_2}-\psi_1\overline{\psi_1}\big).
\end{equation}

Furthermore, we set the Dirac potential $U=V={\widetilde{\epsilon}}e^{w/2}$ for some $\mathbb C'$ valued function $w$ and $\widetilde{\epsilon}=\{\pm 1,\pm i'\}$, and define the the Abresch-Rosenberg
differential $Q\,dz^2=4B\,dz^2$. Following \cite{Kiyohara&Kobayashi_timelike2022}, these can be expressed as
\begin{equation}\label{eq:UV_B for timelike}
\begin{aligned}
\widetilde{\epsilon}e^{w/2}&=U=V=-\frac{H}{2}e^{u/2}+\frac{i'}{4}h,\\
B &= \frac{2H-i'}{4}
    \bigg(2(\psi_1 \overline{\psi_{2}}_z - \overline{\psi_2} {\psi_1}_z\bigr)
    - 4i'H(\psi_1 \overline{\psi_2})^2\bigg)-\frac{\varphi_3^2}{2H-i'}.
    \end{aligned}
\end{equation}
The generating spinors $\tilde{\psi} = (\psi_1, \psi_2)$ satisfy a \emph{linear spinor system} of the form
\begin{equation}\label{eq:linear_spinor_system for timelike}
\tilde\psi_z=\tilde\psi\,\widetilde U,\qquad
\tilde\psi_{\bar z}=\tilde\psi\,\widetilde V,
\end{equation}
where
\[
\widetilde U=\begin{pmatrix}
\frac12 w_z +\frac 1 2 H_z \widetilde{\epsilon} e^{-w/2+u/2} & -\widetilde{\epsilon} e^{w/2}\\[4pt]
B\widetilde{\epsilon}e^{-w/2} & 0
\end{pmatrix},\qquad
\widetilde V=\begin{pmatrix}
0 & -\overline{B}\widetilde{\epsilon}e^{-w/2}\\[4pt]
\widetilde{\epsilon}e^{w/2} & \frac12 w_{\bar z}+\frac 1 2 H_{\bar z} \widetilde{\epsilon} e^{-w/2+u/2}
\end{pmatrix}.
\]
The second column of the first equation, when combined with the first column of the second equation, yields the nonlinear Dirac equations. In particular, letting $\mathfrak{D}'$ denote the Dirac operator for the conformal immersion in $\operatorname{Nil}_t^3$,
\begin{equation}\label{eq:dirac for timelike}\mathfrak{D}'\begin{pmatrix}\psi_1\\\psi_2\end{pmatrix}
:=\begin{pmatrix}
\partial_z\psi_2 + U\psi_1\\[4pt]
-\partial_{\bar z}\psi_1 + V\psi_2
\end{pmatrix}
=
\begin{pmatrix}0\\0\end{pmatrix},
\qquad U=V=\widetilde{\epsilon}e^{w/2}.
\end{equation}
\begin{remark}
 In particular, if the mean curvature
is zero and the function $h$ has positive values, then $\tilde{\epsilon}=i'$ in \eqref{eq:dirac for timelike}.   
\end{remark}

The left-translated unit normal $f^{-1}N$ lies in the de Sitter two sphere $\widetilde{\mathbb S}_1^2\subset\mathfrak{nil}_t^3$,
where
$$
\widetilde{\mathbb S}_1^2=\{x_1\varepsilon_1+x_2\varepsilon_2+x_3\varepsilon_3\in \mathfrak{nil}_t^3|-x_1^2+x_2^2+x_3^3=1\}.
$$
Now, by composing the stereographic projection $\pi_{\mathfrak{nil}_t^3}^+$ from $(0,0,1) \in \widetilde{\mathbb S}_1^2 \subset \mathfrak{nil}_t^3 $ to $\mathbb{C}'$ with the left-translated unit normal $f^{-1}N$ gives the \emph{normal Gauss map}
$g=\pi_{\mathfrak{nil}_t^3}^+\circ(f^{-1}N):M\longrightarrow\mathbb{C}'$. In terms of generating spinors,
\begin{equation}\label{eq:gauss_spinor for timelike}
g=i'\frac{\overline{\psi_1}}{\psi_2},
\qquad
f^{-1}N=\frac{1}{1-g\overline{g}}\big(2\Re(g)\,\varepsilon_1 +2\Im(g)\,\varepsilon_2 - (1+g\overline{g})\,\varepsilon_3\big).
\end{equation}

Next, we recall the family of Maurer-Cartan forms $\alpha^{\mu}$ parameterized by $\mu\in \mathbb S^{1'}:=\{e^{i't}|\,t\in \mathbb R\}$ as follows:
\begin{equation}\label{maurer-cartan-form-timelike}
\alpha^{\mu}=\widetilde{U}^\mu\, dz+ \widetilde{V}^\mu\, d\overline{z}
\end{equation}
where
\[
\widetilde{U}^\mu=\begin{pmatrix}
\frac14 w_z +\frac 1 2 H_z \widetilde{\epsilon} e^{-w/2+u/2} & -\mu^{-1}\widetilde{\epsilon} e^{w/2}\\[4pt]
\mu^{-1}B\widetilde{\epsilon}e^{-w/2} & -\frac14 w_z
\end{pmatrix},\qquad
\widetilde{V}^\mu=\begin{pmatrix}
-\frac14 w_{\bar z} & -\mu\overline{B}\widetilde{\epsilon}e^{-w/2}\\[4pt]
\mu\widetilde{\epsilon}e^{w/2} & \frac14 w_{\bar z}+\frac 1 2 H_{\bar z} \widetilde{\epsilon} e^{-w/2+u/2}
\end{pmatrix}.
\]
Note that $\widetilde U^\lambda|_{\lambda=1}=\widetilde U$ and
$\widetilde V^\lambda|_{\lambda=1}=\widetilde V$ in \eqref{eq:linear_spinor_system for timelike} after suitable gauge transformation applied to $\widetilde{U}$ and $\widetilde{V}$.

\begin{theorem}\cite[Theorem 3.2]{Kiyohara&Kobayashi_timelike2022}
    Let $f$ be a conformal timelike immersion from a simply connected domain $\mathbb D\subset \mathbb C'$ into $\mathrm{Nil_t^3}$ where Dirac potential $(\Re\, U)^2-(\Im \,U)^2\neq 0$. Then the following conditions are  equivalent:
    \begin{enumerate}
        \item $f$ is a timelike minimal surface.
        \item The Dirac potential $\widetilde{\epsilon}e^{w/2}=U=V=-\frac{H}{2}e^{u/2}+\frac{i'}{4}h$ takes purely imaginary values.
        \item $d+\alpha^\mu$ defines a family of flat connections on $\mathbb D\times \mathrm{SU}_{1,1}'$.
    \end{enumerate}
\end{theorem}

From now on, we will assume that the function $h$ takes positive values, that is, the image of the left-translated unit normal $f^{-1}N$ lies in the lower half of the de Sitter two-sphere $\widetilde{\mathbb S}_1^2$. For more details on timelike minimal surfaces in the Lorentzian Heisenberg group, see \cite{Kiyohara&Kobayashi_timelike2022,Kiyohara2025}.

Next, we introduce dual generating spinors for timelike minimal surfaces and show that
they determine dual timelike minimal surface in $\mathrm{Nil}_t^3$.
\subsection{Dual timelike minimal surface} \label{subsec: Dual timelike minimal surface}
We define the pair of dual generating spinors
$\psi_1^{\ast}$ and $\psi_2^{\ast}$ as
\begin{equation}\label{dual_spinors_timelike}
\psi_1^{\ast}
 = {\frac{4\sqrt{\epsilon\, B}}{h}}\;\psi_2,
\qquad
\psi_2^{\ast}
 = {\frac{4\sqrt{\overline{\epsilon}\,\overline{B}}}{h}}\;\psi_1,
\end{equation}
where $B$ is the Abresch-Rosenberg differential and $h$ is the support
function as in Subsection~\ref{subsec:Preliminaries_timelike_minimal_surfaces}. If $B\Bar{B}\equiv 0$, the corresponding timelike minimal surface is called a null scroll. We exclude null scroll from our study (for details see \cite{Kiyohara2025}). 

The following lemma shows that the dual generating spinors satisfy a linear system.
\begin{lemma}\label{lemma:timilike}
    Let $\psi_1$ and $\psi_2$ be the generating spinors of a timelike minimal
surface $f$ in $\mathrm{Nil}_t^3$, and let $\psi_1^{\ast}$ and $\psi_2^{\ast}$ be the dual
generating spinors defined in~\eqref{dual_spinors_timelike}.  
Assume that the Abresch-Rosenberg differential satisfies $B\bar{B}\not\equiv0$.
Then the vector of generating spinors $\tilde\psi^{\ast}=(\psi_1^{\ast},\psi_2^{\ast})$
satisfies the following linear spinor system:
\begin{equation}\label{eq:UandV_timelike}
  \tilde\psi^{\ast}_z = \tilde\psi^{\ast}\,\widetilde U^{\ast},
  \qquad
  \tilde\psi^{\ast}_{\bar z} = \tilde\psi^{\ast}\,\widetilde V^{\ast},
\end{equation}
where
\[
\widetilde U^{\ast} =
\begin{pmatrix}
\tfrac12 w^{\ast}_z & i'\epsilon e^{w^{\ast}/2}\\[3pt]
-i'\epsilon B^{\ast}e^{-w^{\ast}/2} & 0
\end{pmatrix},
\qquad
\widetilde V^{\ast} =
\begin{pmatrix}
0 & i'\Bar{\epsilon}\Bar{B}^{\ast}e^{-w^{\ast}/2}\\[3pt]
i' \Bar{\epsilon}e^{w^{\ast}/2} & \tfrac12 w^{\ast}_{\bar z}
\end{pmatrix},
\]
with
\begin{equation}\label{eq:B&Dicar_potential_dual_timetikeminimal}
e^{w^{\ast}/2} = \frac{4\sqrt{\epsilon\Bar{\epsilon}B\overline{B}}}{ h},\qquad B^{\ast}=B.
\end{equation}
\end{lemma}

\begin{proof} Recall that for $\epsilon \in \{\pm 1, \pm i'\}$, we have $1/\epsilon = \epsilon$.
From~\eqref{eq:B&Dicar_potential_dual_timetikeminimal}, and the relations
\[
\frac12 w^{\ast}_z
   = \Bigl(\log{\frac{4\sqrt{\epsilon B}}{h}}\Bigr)_z,
   \qquad
   \frac{ h B}{4\sqrt{\epsilon\Bar{\epsilon}B\Bar{B}}}
   = B^{\ast}\ e^{-w^{\ast}/2}, \qquad e^{w/2}=\frac{h}{4},
\]
where $B(z)\overline{B(z)}\neq 0$, we compute
\begin{align*}
\psi^{\ast}_{1z}
 &=\Bigl({\frac{4\sqrt{\epsilon B}}{h}}\Bigr)_z\psi_2
   +{\frac{4\sqrt{\epsilon B}}{h}}\;\psi_{2z}\\
 &=\Bigl(\log{\frac{4\sqrt{\epsilon B}}{h}}\Bigr)_z\psi_1^{\ast}
  -\frac{ \epsilon h B i'}{4\sqrt{\epsilon\Bar{\epsilon}B\Bar{B}}}\psi_2^{\ast}\\
  &=\frac 12 w_z^\ast\psi_1^\ast-i'\epsilon B^\ast e^{-w^\ast/2}\psi_2^\ast.
\end{align*}
Similarly,
\begin{align*}
\psi^{\ast}_{2z}
 &=\Bigl(\frac{4\sqrt{\Bar{\epsilon} \Bar{B}}}{h}\Bigr)_z\psi_1
   +\frac{4\sqrt{\Bar{\epsilon} \Bar{B}}}{h}\;\psi_{1z}\\
 &=-\!(\log h)_z\frac{4\sqrt{\Bar{\epsilon} \Bar{B}}}{h}\psi_1
   +\frac{4\sqrt{\Bar{\epsilon} \Bar{B}}}{h}\Bigl(\tfrac12 w_z\psi_1+i'Be^{-w/2}\psi_2\Bigr)\\
   &=i'\epsilon e^{w^{\ast}/2}\psi^{\ast}_1.
\end{align*}
A similar computation for $\tilde\psi^{\ast}_{\bar z}$ gives the second equation
in~\eqref{eq:UandV_timelike}. 
\end{proof}
With the above lemma, we state the main theorem for this section. 
\begin{theorem}\label{thm 1 for timelike}
Retain all the notation and the assumptions of Lemma~\ref{lemma:timilike} and the relation in \eqref{eq:B&Dicar_potential_dual_timetikeminimal}, there exists a conformal  immersion $f^{\ast}$ in $\mathrm{Nil}_t^3$ whose generating
spinors are $\psi_1^{\ast}$ and $\psi_2^{\ast}$, with Dirac potential
$U^{\ast}=V^{\ast}=i'e^{w^{\ast}/2}$, and whose left-translated unit normal equals
$f^{-1}N$ up to sign.  
Moreover, the following hold:
\begin{enumerate}
\item By choosing the left-translated unit normal $N^{\ast}$ so that
      $(f^{\ast})^{-1}N^{\ast}=f^{-1}N\subset\mathfrak{nil}_t^3$, the metric $e^{u^{\ast}}dzd\bar{z}$, support function $h^{\ast}$,  Abresch-Rosenberg differential $B^{\ast}dz^2$, and normal Gauss map $g^{\ast}$ of $f^{\ast}$ are respectively given by
      \begin{equation}\label{eq:ehBG_dualtimelike}
      e^{u^{\ast}/2}
        = \frac{16\sqrt{\epsilon\Bar{\epsilon}B\Bar{B}}}{h^2}\,e^{u/2},
      \qquad
      h^{\ast} = \frac{16\sqrt{\epsilon\Bar{\epsilon}B\Bar{B}}}{h},
      \qquad
      B^{\ast}=B,
      \qquad\text{and}\quad
      g^{\ast}={g}.
      \end{equation} 
\item The surface $f^{\ast}$ is non-vertical and timelike minimal.
\item It satisfies the duality $f^{\ast\ast}=f$ up to a rigid motion.

\end{enumerate}
The surface $f^{\ast}$ will be called the \emph{dual timelike minimal surface} to $f$.
\end{theorem}

\begin{proof}  
From Lemma~\ref{lemma:timilike}, $\tilde\psi^{\ast}=(\psi_1^{\ast},\psi_2^{\ast})$ defines a conformal immersion
$f^{\ast}$ in $\mathrm{Nil}_t^3$.  
Since tangent planes of $f$ and $f^{\ast}$ coincide, their left-translated unit normals $f^{-1}N$ and $(f^{\ast})^{-1}N^{\ast}$ agree
up to sign.

(1)  The metric in \eqref{eq:ehBG_dualtimelike} is clear from the definition of
$\psi_1^{\ast},\psi_2^{\ast}$ and  $h^\ast$ follows from the definition of the support function of $f^\ast$. We choose such Gauss map $(f^{\ast})^{-1}N^{\ast}=f^{-1}N$ that gives $g^\ast=g.$  

(2)  From construction $f^\ast$ is non-vertical and $H^{\ast}=0$ follows from the Dirac potential
\[
U^{\ast}=V^{\ast}=i'e^{w^{\ast}/2}
  =-\frac{H^{\ast}}{2}e^{u^{\ast}/2}
    +\frac{i'}{4}h^{\ast}
  = \frac{4i'\sqrt{\epsilon \Bar{\epsilon} B \overline{B}}}{h}                        .
\]
Thus $f^{\ast}$ is non-vertical timelike minimal surface in $\operatorname{Nil}_t^3$.

(3)  The double-dual spinors are
$$
\psi^{\ast\ast}_1
 = {\frac{4\sqrt{\epsilon\, B^{\ast}}}{h^{\ast}}}\,\psi^{\ast}_2=\psi_1,
\qquad
\psi^{\ast\ast}_2
 = {\frac{4\sqrt{\Bar{\epsilon}\, \overline{B}^{\ast}}}{h^{\ast}}}\,\psi^{\ast}_1=\psi_2.
$$
Hence $(\psi^{\ast\ast}_1,\psi^{\ast\ast}_2)$ defines the original surface $f$ up to rigid
motion.
\end{proof}


Next, we present the Sym formula for dual timelike minimal surfaces in
$\mathrm{Nil}_t^3$.

\subsection{The Sym formula for dual timelike minimal surface}\label{subsec: The Sym formula for dual timelike minimal surface}

We first identify the Lie algebra $\mathfrak{nil_t^3}$ with  $\mathfrak{su}_{1,1}'$ (Lie algebra of $\operatorname{SU}'_{1,1}$)
as a real vector space. In $\mathfrak{su}_{1,1}'$, choose the basis
\begin{equation}\label{eq:basis_su_11}
\mathcal E_1=\frac12\begin{pmatrix}0&-i'\\i'&0\end{pmatrix},\qquad
\mathcal E_2=\frac12\begin{pmatrix}0&-1\\-1&0\end{pmatrix},\qquad
\mathcal E_3=\frac12\begin{pmatrix}i'&0\\0&-i'\end{pmatrix}.
\end{equation}
We can see that $\{\mathcal E_1,\mathcal E_2,\mathcal E_3\}$ is an orthogonal basis of $\mathfrak{su}_{1,1}'$. A linear isomorphism $\Xi_t:\mathfrak{su}_{1,1}'\longrightarrow\mathfrak{nil}_t^3$
is then given by
\begin{equation}\label{identifacatin_timelike}
\mathfrak{su}_{1,1}'\ni x_1\mathcal E_1 + x_2\mathcal E_2 + x_3\mathcal E_3 \longmapsto x_1 \varepsilon_1 + x_2 \varepsilon_2 + x_3 \varepsilon_3\in\mathfrak{nil}_t^3.
\end{equation}
Define a smooth bijection $\Xi_{\mathrm{nil}_t}:\mathfrak{su}_{1,1}'\longrightarrow\mathrm{Nil}_t^3$ by
$\Xi_{\mathrm{nil}_t}:=\exp\circ\Xi_t$
$$
\Xi_{\mathrm{nil}_t}(x_1\mathcal E_1 + x_2\mathcal E_2 + x_3\mathcal E_3)=(x_1, x_2, x_3).
$$
For timelike minimal surface with $h>0$, the Maurer-Cartan forms $\alpha^{\mu}$ as in \eqref{maurer-cartan-form-timelike} rewrite as
\begin{equation}\label{maurer-cartan-form-timelike_minimal}
\alpha^{\mu}={U}^\mu\, dz+ {V}^\mu\, d\overline{z}
\end{equation}
where
\[
{U}^\mu=\begin{pmatrix}
\frac14 w_z  & -\mu^{-1}i' e^{w/2}\\[4pt]
\mu^{-1}Bi'e^{-w/2} & -\frac14 w_z
\end{pmatrix},\qquad
{V}^\mu=\begin{pmatrix}
-\frac14 w_{\bar z} & -\mu\overline{B}i'e^{-w/2}\\[4pt]
\mu i'e^{w/2} & \frac14 w_{\bar z}
\end{pmatrix}.
\]
Let $F$ be the
corresponding $\mathrm{SU}_{1,1}$ valued solution of the equation $F^{-1}dF=\alpha^\mu$
$(\mu\in \mathbb S^{1'})$, where $\alpha^\mu$ is defined by \eqref{maurer-cartan-form-timelike_minimal}. Then $F$ is called
the \emph{extended frame of the timelike minimal surface} $f$. Also, $F$ can be represented as:
\begin{equation}\label{eq:extendedframe_timelike}
F(\mu)=\frac{1}{\sqrt{\psi_2(\mu)\overline{\psi_2(\mu)}-\psi_1(\mu) \overline{\psi_1(\mu)}   }}
\begin{pmatrix}
   \overline{\psi_2(\mu)} & \overline{\psi_1(\mu)}\\[6pt]
\psi_1(\mu) & \psi_2(\mu) 
\end{pmatrix},
\end{equation}
where $\psi_1(\mu)$ and $\psi_2(\mu)$ denote families of functions such that
$\psi_1(\mu=1)=\psi_1$ and $\psi_2(\mu=1)=\psi_2$, which are the generating spinors
of the surface $f$.

We now prove the Sym formula for the dual of the timelike minimal surfaces.
\begin{theorem}\label{thm:Sym_formula_dualtimelike}
Let $F$ be the extended frame for a timelike minimal surface $f$ on $\mathbb D$ with
$B\not\equiv0$ and $\epsilon=i'$. Define $m_\pm$ and $N_m$ respectively by
$$
m_\pm = -i'\,\mu(\partial_\mu F)F^{-1}\pm N_m,\qquad
N_m=\frac{i'}{2}\operatorname{Ad}(F)\sigma_3,
$$
where $\sigma_3=\begin{pmatrix}
    1& 0\\
    0& -1
\end{pmatrix}.$ Moreover, define  maps $f^\mu_\pm:\mathbb D\longrightarrow\mathrm{Nil}_t^3$ by $f^\mu_\pm:=\Xi_{\mathrm{nil_t}}\circ\widehat f^\mu_\pm$
with
\begin{equation}\label{eq:sym_formula_timelike}
\widehat f^\mu_\pm=
\Bigl(m_\pm^{\mathrm{o}} - \frac{i'}{2}\mu(\partial_\mu m_\pm)^{\mathrm{d}}\Bigr)\Big|_{\mu\in \mathbb S^1},
\end{equation}
where superscripts ``o'' and ``d'' denote off-diagonal and diagonal parts, respectively.
Then, for each $\mu\in \mathbb S^{1'}$, the maps $f^\mu_\pm$ are timelike minimal surfaces in $\mathrm{Nil}_t^3$
and $N_m$ is the normal Gauss map of $f^\mu_\pm$. In particular, $f^\mu_-|_{\mu=1}$
is the original timelike  minimal surface $f$ up to a rigid motion and $f^\mu_+|_{\mu=1}$ is the
dual timelike minimal surface $f^\ast$ of $f$. Moreover, for each $\mu\in \mathbb S^{1'}$, $f^\mu_-$ and $f^\mu_+$ are dual to each other.
\end{theorem}
\begin{proof}
From \cite[Theorem 4.1]{Kiyohara&Kobayashi_timelike2022}, it is clear that $f^\mu_-|_{\mu=1}$ is a timelike minimal surface in $\operatorname{Nil}_t^3$.
We only show for $f^\mu_+|_{\mu=1}$, and both are dual to each other for each $\mu\in \mathbb S^{1'}$. From a direct calculation, we have
$$
\partial_z m_+ = \operatorname{Ad}(F)\Bigl(-i'\mu\partial_\mu U^\mu + \frac{i'}{2}[U^\mu,\sigma_3]\Bigr)
= 2\mu^{-1} B e^{-w/2}\operatorname{Ad}(F)\begin{pmatrix}0&0\\1&0\end{pmatrix}.
$$
Using the generating spinors and the support function $h$, this becomes
\[
\partial_z m_+ = \frac{16\mu^{-1}B}{h^2}
\begin{pmatrix}
\psi_2(\mu)\overline{\psi_1(\mu)} & -\overline{\psi_1(\mu)^2}\\[6pt]
 \psi_2(\mu)^2& -\psi_2(\mu)\overline{\psi_1(\mu)}
\end{pmatrix}.
\]
From the dual
generating spinors \eqref{dual_spinors_timelike}, 
we have
\begin{equation}\label{eq:partial_m+_timelike}
\partial_z m_+ = \varphi_1^+(\mu)\mathcal E_1 + \varphi_2^+(\lambda)\mathcal E_2 + i\,\varphi_3^+(\mu)\mathcal E_3,
\end{equation}
where
\begin{align*}
\varphi_1^+(\mu)&=\mu^{-1}\epsilon\, i'\,\bigl((\overline{\psi_2^\ast(\mu)})^2 + (\psi_1^\ast(\mu))^2\bigr),\\
\varphi_2^+(\mu)&=\mu^{-1}\epsilon\,\bigl((\overline{\psi_2^\ast(\mu)})^2 -(\psi_1^\ast(\mu))^2\bigr),\\
\varphi_3^+(\mu)&=2\mu^{-1}\epsilon\,\psi_1^\ast(\mu)\overline{\psi_2^\ast(\mu)}.
\end{align*}
A computation of $\partial_z(i\mu\partial_\mu m_+)$ combined with the diagonal part
analysis yields
\begin{equation}\label{eq:2ndpart_symformula_timelike}
    \begin{aligned}
\partial_z\Bigl(-\tfrac{i'}{2}\mu(\partial_\mu m_+)\Bigr)^d
&=\frac{i'}{2}(\partial_z m_+)^d+\frac{1}{2}[m_+-N_m,\,\partial_z m_+]^d\\
&=\Bigl(\varphi_3^+(\mu) -\frac12\varphi_1^+(\mu)\int\varphi_2^+(\mu)\,dz
 + \frac12\varphi_2^+(\mu)\int\varphi_1^+(\mu)\,dz\Bigr)\mathcal E_3.
\end{aligned}
\end{equation}
Combining \eqref{eq:partial_m+_timelike} and \eqref{eq:2ndpart_symformula_timelike} we obtain
\[
\partial_z\widehat f^\mu_+ = \varphi_1^+(\mu)\mathcal E_1 + \varphi_2^+(\mu)\mathcal E_2
+ \Bigl(\varphi_3^+(\mu) - \frac12\varphi_1^+(\mu)\int\varphi_2^+(\mu)\,dz
 + \frac12\varphi_2^+(\mu)\int\varphi_1^+(\mu)\,dz\Bigr)\mathcal E_3.
\]
Using the identification \eqref{identifacatin_timelike} and the left translation by $(f^\mu_+)^{-1}$, we obtain
$$
(f^\mu_+)^{-1}\partial_z f^\mu_+ = \varphi_1^+(\mu)\varepsilon_1 + \varphi_2^+(\mu)\varepsilon_2 + \varphi_3^+(\mu)\varepsilon_3.
$$
It is clear that $\psi_1^\ast(\mu)$ and $\psi_2^\ast(\mu)$ are spinors for $f^\mu_+$ for
each $\mu\in \mathbb S^1$. In particular, the function
\[
\frac{1}{2}\bigg(\psi_2^*\,\overline{\psi_2^*}-\psi_1^*\,\overline{\psi_1^*}\bigg)= e^{w^\ast/2} =\frac{4\sqrt{\epsilon\Bar{\epsilon}B\Bar{B}}}{h}
\]
does not depend on $\mu$ and so the mean curvature $H=0$ ($i' e^{w^*/2}\in i'\mathbb C'$, for each $\mu$). Moreover, the conformal factor of the induced metric of $f^\lambda_+$ is given by
\[
e^{u^\ast}=4(\psi_1^*\,\overline{\psi_1^*}+\psi_2^*\,\overline{\psi_2^*})^2 = \frac{16\sqrt{\epsilon\Bar{\epsilon}B\Bar{B}}}{h^2}\,e^u,
\]
which is degenerate only at $B(z)\overline{B(z)}=0$. Thus, $f^\mu_+$ defines a timelike minimal surface in
$\mathrm{Nil}_t^3$ for each $\mu\in \mathbb S^1$ where $B(z)\overline{B(z)}\neq 0.$ This completes the proof.
\end{proof}
\textbf{Example:} Consider two hyperbolic paraboloids $x_3=\tfrac{x_1x_2}{2}$ and $x_3=-\tfrac{x_1x_2}{2}$ given as in \cite[Section 6]{Kiyohara&Kobayashi_timelike2022} and their corresponding holomorphic potentials defined as
\[
\zeta_1 = \mu^{-1}
\begin{pmatrix}
0 & -\dfrac{i'}{4} \\
\dfrac{i'}{4} & 0
\end{pmatrix} dz,\quad \text{ and }\quad\zeta_2 = \mu^{-1}
\begin{pmatrix}
0 & -\dfrac{i'}{4} \\
-\dfrac{i'}{4} & 0
\end{pmatrix} dz.
\]

Let $\Phi_1$ and $\Phi_2$ be the solutions of the ODEs $d\Phi_1=\Phi_1\zeta_1$ and $d\Phi_2=\Phi_2\zeta_2$ with the initial conditions $\Phi_1(z=0)=\mathrm{id}$ and $\Phi_2(z=0)=\mathrm{id}$, respectively.
Then, from \cite[Section 6]{Kiyohara&Kobayashi_timelike2022}, the extended frames $F_1$ and $F_2$ from the Iwasawa decompositions $\Phi_1=F_1V_1$ and $\Phi_2=F_2V_2$, respectively, are given by
\begin{align*}
F_1 &=
\begin{pmatrix}
\cos \dfrac{\mu^{-1}z+\mu\bar z}{4} &
- i' \sin \dfrac{\mu^{-1}z+\mu\bar z}{4} \\[10pt]
i' \sin \dfrac{\mu^{-1}z+\mu\bar z}{4} &
\cos \dfrac{\mu^{-1}z+\mu\bar z}{4}
\end{pmatrix},\\
F_2 &=
\begin{pmatrix}
\cosh \dfrac{-\mu^{-1}z+\mu\bar z}{4} &
i' \sinh \dfrac{-\mu^{-1}z+\mu\bar z}{4} \\[10pt]
i' \sinh \dfrac{-\mu^{-1}z+\mu\bar z}{4} &
\cosh \dfrac{-\mu^{-1}z+\mu\bar z}{4}
\end{pmatrix}.
\end{align*}

By applying Sym formula as in \eqref{eq:sym_formula_timelike} to both extended frames $F_1$ and $F_2$, the respective family of timelike minimal surfaces $f_{1-}^\mu$ and $f_{2-}^\mu$  in $\mathrm{Nil}_t^3$ are given by
\begin{align*}
    f_{1-}^\mu&=
\left(
i' \frac{\mu^{-1}z-\mu\bar z}{2},
\, \sin \frac{\mu^{-1}z+\mu\bar z}{2},
\, i' \frac{\mu^{-1}z-\mu\bar z}{4}
\sin \frac{\mu^{-1}z+\mu\bar z}{2}
\right),\\
f_{2-}^\mu& =
\left(
- \sinh \frac{i'(-\mu^{-1}z+\mu\bar z)}{2},
\, \frac{\mu^{-1}z+\mu\bar z}{2},
\, \frac{\mu^{-1}z+\mu\bar z}{4}
\sinh \frac{i'(-\mu^{-1}z+\mu\bar z)}{2}
\right).
\end{align*}
For each $\mu=e^{i't}\in\mathbb S_1^1,$ $f_{1-}^\mu$ and $f_{2-}^\mu$ are parts of the hyperbolic paraboloids $x_3=\tfrac{x_1x_2}{2}$ and $x_3=-\tfrac{x_1x_2}{2}$ respectively. Using the Sym formula as in \eqref{eq:sym_formula_timelike}, the corresponding dual surfaces are
\begin{align*}
    f_{1+}^\mu&=
\left(
i' \frac{\mu^{-1}z-\mu\bar z}{2},
\, -\sin \frac{\mu^{-1}z+\mu\bar z}{2},
\, -i' \frac{\mu^{-1}z-\mu\bar z}{4}
\sin \frac{\mu^{-1}z+\mu\bar z}{2}
\right),\\
f_{2+}^\mu& =
\left(
 \sinh \frac{i'(-\mu^{-1}z+\mu\bar z)}{2},
\, \frac{\mu^{-1}z+\mu\bar z}{2},
\, -\frac{\mu^{-1}z+\mu\bar z}{4}
\sinh \frac{i'(-\mu^{-1}z+\mu\bar z)}{2}
\right).
\end{align*}
Then it is clear that $f_{1+}^\mu$ and $f_{2+}^\mu$ are the parts of the same hyperbolic paraboloids $x_3=\tfrac{x_1x_2}{2}$ and $x_3=-\tfrac{x_1x_2}{2}$ respectively. Therefore, both timelike minimal surfaces $f_{1-}^\mu$ and $f_{2-}^\mu$ are self-dual.

\bibliography{ref.bib}

\end{document}